\documentclass[10pt, reqno]{amsart}

\usepackage{mathrsfs}
\usepackage{mathtools}
\usepackage{enumerate}
\usepackage{url}
\usepackage{hyperref}
\usepackage{tikz}
\newcommand*\circled[1]{\tikz[baseline=(char.base)]{
            \node[shape=circle,draw,inner sep=2pt] (char) {#1};}}

\newtheorem{theorem}{Theorem}
\newtheorem{lemma}[theorem]{Lemma}

\newtheorem{corollary}[theorem]{Corollary}

\newtheorem*{theorem*}{Theorem}

\theoremstyle{definition}

\numberwithin{theorem}{section}
\numberwithin{equation}{section}

\newcommand{\F}{\mathbf{F}}

\newcommand{\Z}{\mathbf{Z}}

\renewcommand{\P}{\mathbf{P}}

\renewcommand{\S}{\mathfrak{S}}

\renewcommand{\epsilon}{\varepsilon}
\newcommand{\eps}{\epsilon}

\renewcommand{\hat}{\widehat}
\renewcommand{\(}{\left(}
\renewcommand{\)}{\right)}

\newcommand{\tx}[1]{{\textstyle #1}}

\newcommand{\sump}{\sideset{}{'}\sum}
\newcommand{\floor}[1]{\lfloor#1\rfloor}
\newcommand{\norm}[1]{\left\|#1\right\|}
\newcommand{\tr}{\text{tr}}

\renewcommand{\Sigma}[1]{\tx\sum #1}

\begin{document}

\title{More on additive triples of bijections}



\author{Sean Eberhard}
\address{Sean Eberhard, London, UK}
\email{eberhard.math@gmail.com}

\begin{abstract}
We study additive properties of the set $S$ of bijections (or permutations) $\{1,\dots,n\}\to G$, thought of as a subset of $G^n$, where $G$ is an arbitrary abelian group of order $n$. Our main result is an asymptotic for the number of solutions to $\pi_1 + \pi_2 + \pi_3 = f$ with $\pi_1,\pi_2,\pi_3\in S$, where $f:\{1,\dots,n\}\to G$ is an arbitary function satisfying $\sum_{i=1}^n f(i) = \sum G$. This extends recent work of Manners, Mrazovi{\'c}, and the author~\cite{EMM}. Using the same method we also prove a less interesting asymptotic for solutions to $\pi_1 + \pi_2 + \pi_3 + \pi_4 = f$, and we also show that the distribution $\pi_1+\pi_2$ is close to flat in $L^2$.

As in \cite{EMM}, our method is based on Fourier analysis, and we prove our results by carefully carving up $\hat{G}^n$ and bounding various character sums. This is most complicated when $G$ has even order, say when $G = \F_2^d$.

At the end of the paper we explain two applications, one coming from the Latin squares literature (counting transversals in Latin hypercubes) and one from cryptography (PRP-to-PRF conversion).
\end{abstract}

\maketitle

\section{Introduction}

Let $G$ be an abelian group of order $n$, and let $S$ be the set of bijections $\pi:\{1,\dots,n\}\to G$. We are interested in additive properties of $S$ as a subset of $G^n$. For example, we are interested in counting solutions to
\[
  \pi_1 + \pi_2 = \pi_3
  \qquad(\pi_1,\pi_2,\pi_3\in S).
\]
The solutions to this equation are called \emph{additive triples} (or \emph{Schur triples}). The following estimate was the main theorem of \cite{EMM}.

\begin{theorem}[Main theorem of \cite{EMM}]\label{main-EMM}
Assume $n$ is odd. Then the the number of additive triples in $S$ is
\[
  (e^{-1/2} + o(1)) n!^3/n^{n-1}.
\]
\end{theorem}

The odd-order hypothesis here is irritating. In \cite{EMM} our main interest was the cyclic group $G = \Z/n\Z$, and in that case the restriction to odd $n$ is natural, since otherwise there aren't any solutions to $\pi_1 + \pi_2 = \pi_3$, but in general it is a shortcoming. It is a little sad not to have an asymptotic in the case $G = \F_2^d$, for example. Here, among other things, we make amends by proving the following theorem.

\begin{theorem}\label{main-even}
Assume $\Sigma G = 0$. Then the number of additive triples in $S$ is
\[
  (e^{-1/2}+o(1))n!^3/n^{n-1}.
\]
\end{theorem}

Here and throughout the paper we write simply $\Sigma G$ for $\sum_{x\in G} x$, which is a characteristic element of $G$. It is easy to see that $\Sigma G$ is equal to the unique element of order $2$ in $G$ if there is one, and otherwise zero. Thus $\Sigma G \neq 0$ if and only if $G$ has the form $\Z/m\Z \times H$ with $m$ even and $|H|$ odd. In these groups there are no additive triples in $S$ (since $\pi_1 + \pi_2 = \pi_3$ implies $2\Sigma G = \Sigma G$), and Theorem~\ref{main-even} covers all other groups.

Generalizing the problem, suppose we count solutions to
\[
  \pi_1 + \pi_2 + \pi_3 = f,
\]
where $f:\{1,\dots,n\}\to G$ is a fixed but arbitrary function. (The previous problem is the case $f=0$, since $\pi$ is a bijection if and only if $-\pi$ is a bijection.) Call such solutions \emph{additive $f$-triples}. The natural hypothesis about $f$ and $G$ is that $\sum_{i=1}^n f(i) = \Sigma G$. Our main theorem is the following generalization of Theorem~\ref{main-even}.

\begin{theorem}\label{main-f}
Let $f:\{1,\dots,n\}\to G$ be a function satisfying $\sum_{i=1}^n f(i) = \Sigma G$. Then the number of additive $f$-triples in $S$ is
\[
  (\S(f) + o(1)) n!^3/n^{n-1},
\]
where
\[
  \S(f) = \exp\( - \frac1{2n^2} \sum_{x\in G} |f^{-1}(x)|^2\).
\]
\end{theorem}

By analogy with analytic number theory, we might call $\S(f)$ the \emph{singular series} associated with $f$. The expression inside the brackets is closely related to the so-called \emph{collision entropy} (or \emph{R{\'e}nyi entropy}) of $f$, defined by
\[
  H_2(f) = -\log\sum_{x\in G}\P(f = x)^2.
\]
Explicitly,
\[
  \S(f) = \exp\(-\frac12 e^{-H_2(f)}\).
\]
The value of $\S(f)$ can be as small as $e^{-1/2}$ (when $f$ is constant), but generically $\S(f) = 1 + o(1)$. Thus we can describe the distribution of $\pi_1 + \pi_2 + \pi_3$ as pretty close to uniform on the coset $\sum_{i=1}^n f(i) = \Sigma G$, but with a slight aversion to functions with high collision entropy.

If we have four or more $\pi$ summands, then distribution becomes asymptotically flat. In fact this follows already from Theorem~\ref{main-f}, but by directly applying our method we will prove the following quantitative version of this assertion.

\begin{theorem}\label{thm:main-4}
Let $d\geq4$ be an integer, and let $f:\{1,\dots,n\}\to G$ a function satisfying $\sum_{i=1}^n f(i) = d \sum G$. Then the number of solutions to
\[
  \pi_1 + \cdots + \pi_d = f
  \qquad(\pi_1,\dots,\pi_d\in S)
\]
is
\[
  \(1 + O_d(n^{3-d})\) n!^d/n^{n-1}.
\]
\end{theorem}

Finally, suppose we have just two bijections $\pi_1$ and $\pi_2$, and we consider $\pi_1+\pi_2$. In the case of just two bijections, the distribution of is very far from flat in any $L^\infty$ sense. Indeed, it follows from Theorem~\ref{main-even} and a symmetry argument that (provided $\Sigma G = 0$) the number of solutions to $\pi_1 + \pi_2 = \pi$, where $\pi$ is any fixed bijection, is $(e^{-1/2}+o(1))n!^2/n^{n-1}$, while clearly the number of solutions to $\pi_1 + \pi_2 = 0$ is $n!$. However, we can at least prove that $\pi_1 + \pi_2$ is close to uniform (on the coset $\sum_{i=1}^n f(i) = 0$) in $L^2$. The following theorem is a slight generalization.

\begin{theorem}\label{thm:L2}
Let $m<n$ be an integer, and let
\[
  \pi_1,\pi_2:\{1,\dots,m\}\to G
\]
be uniformly random injections. Then the $L^2$ distance between the distribution of $\pi_1 + \pi_2$ and the uniform distribution is $O(m/n^{3/2})$, where the $L^2$ norm is taken with respect to the uniform distribution. In other words, if $S_m$ is the set of injections then
\[
  \norm{\(\frac{n^m(n-m)!}{n!}\)^2 1_{S_m} * 1_{S_m} - 1 }_2 \leq O\(m/n^{3/2}\).
\]
Thus in particular the total variation distance between the distribution of $\pi_1 + \pi_2$ and the uniform distribution is $O(m/n^{3/2})$.
\end{theorem}

At the end of the paper we give two applications. The first is to counting transversals in Latin hypercubes. Just as counting solutions to $\pi_1 + \pi_2 = \pi_3$ is equivalent to counting transversals in a certain Latin square, so counting solutions to $\pi_1 + \cdots + \pi_d = \pi_{d+1}$ is equivalent to counting transversals in a certain Latin hypercube. Thus our estimates answer some (modest) questions in the Latin squares literature.

The second application comes from cryptography, and has to do with conversion of pseudorandom permutations (PRPs) to pseudorandom functions (PRFs). A common construction is to take two pseudorandom permutations $\{0,1\}^d \to \{0,1\}^d$ and to use their bitwise xor. The security of this construction is closely related to Theorem~\ref{thm:L2} with $G = \F_2^d$. In the cryptographers' language, Theorem~\ref{thm:L2} implies that the advantage to an adversary with access to at most $m$ queries is at most $O(m/2^{3d/2})$, as long as $m<2^d$.

\section{Overview of the paper}\label{sec:recap}

As usual in additive combinatorics, our main tool for counting additive triples is the Fourier transform. We fix our conventions by giving $G$ the uniform measure and $\hat{G}$ the counting measure. Thus the Fourier transform $\hat{1_S}$ of $1_S$ is defined by
\[
  \hat{1_S}(\chi) = \frac1{n^n} \sump_{x_1,\dots,x_n} \chi_1(x_1) \cdots \chi_n(x_n)
  \qquad (\chi = (\chi_1,\dots,\chi_n) \in \hat{G}^n),
\]
where the sum is primed to indicate that only distinct $x_1,\dots,x_n\in G$ are to be considered. The number of additive $f$-triples in $S$ can then be expressed as
\[
 n^{2n} 1_S*1_S*1_S(f)
  = n^{2n} \sum_{\chi\in\hat{G}^n} \hat{1_S}(\chi)^3 \chi(f).
\]

In \cite{EMM} we estimated this sum (in the case $f=0$, in which case $\chi(f)=1$ identically) by distinguishing various regions of $\hat{G}^n$ depending on the \emph{entropy} $H(\chi)$ of $\chi\in\hat{G}^n$, which we defined as follows. Suppose $\chi$ has $a_1$ coordinates equal to $\psi_1$, $a_2$ coordinates equal to $\psi_2$, and so on, where $\psi_1, \dots, \psi_k$ are distinct and $\sum_{i=1}^k a_i = n$. Since $\hat{1_S}(\chi)$ is a totally symmetric function of $\chi_1, \dots, \chi_n$, it makes sense to denote $\hat{1_S}(\chi)$ with the simple shorthand
\[
  \hat{1_S}(\psi_1^{a_1}, \dots, \psi_k^{a_k}).
\]
The number of $\chi$ so represented is
\[
  \binom{n}{a_1, \dots, a_k}.
\]
We define $H(\chi)$ by
\[
  H(\chi) = \frac1n \log \binom{n}{a_1,\dots,a_k}.
\]
The language is motivated by the fact that $H(\chi)$ is approximately the (Shannon) entropy of the random variable which takes value $i$ with probability $a_i/n$:
\[
  H(\chi) \approx \sum_{i=1}^k \frac{a_i}n \log \frac{n}{a_i}.
\]
(See Lemma~\ref{entropy} for a precise version of this approximation.) Our basic division of $\hat{G}^n$ is then defined by $H \leq \epsilon$ and $H > \epsilon$, where $\epsilon$ is a small positive parameter:
\[
  \sum_{\chi\in\hat{G}^n} \hat{1_S}(\chi)^3 \chi(f)
  = \sum_{H(\chi) \leq \epsilon} \hat{1_S}(\chi)^3 \chi(f) + \sum_{H(\chi)>\epsilon} \hat{1_S}(\chi)^3 \chi(f).
\]

The sum over low-entropy characters can in turn be related to a sum over \emph{sparse} characters, where we call $\chi$ $m$-sparse if exactly $m$ of its coordinates are nonzero. First, by a straightforward calculation, if $H(\chi)\leq \epsilon$ then some coordinate of $\chi$ is repeated at least ${(1-2\epsilon)n}$ times. Second, assuming $\sum_{i=1}^n f(i) = \Sigma G$, the expression $\hat{1_S}(\chi)^3 \chi(f)$ is invariant under global shifts of $\chi$: if $\chi'_i = \chi_i + \psi$ for each $i$ then
\begin{align}
  \hat{1_S}(\chi') &= \frac1{n^n} \sump_{x_1,\dots,x_n} \chi_1(x_1) \cdots \chi_n(x_n) \psi(x_1 + \dots + x_n)\nonumber\\
  &= \hat{1_S}(\chi) \psi\(\Sigma G\),\label{shift-invariance}
\end{align}
so, since $2\Sigma G = 0$,
\[
  \hat{1_S}(\chi')^3 \chi'(f) = \hat{1_S}(\chi)^3 \chi(f) \psi(3\Sigma G + \Sigma f) = \hat{1_S}(\chi)^3 \chi(f).
\]
Thus, by always shifting so that the majority coordinate is $0$, we can decompose the sum over low-entropy characters as
\[
  \sum_{H(\chi)\leq \epsilon} \hat{1_S}(\chi)^3 \chi(f) = n \sum_{m=0}^{2\epsilon n} \sum_{\substack{m\text{-sparse}~\chi\\H(\chi)\leq \epsilon}} \hat{1_S}(\chi)^3 \chi(f).
\]
Thus we have
\begin{align}
  \sum_{\chi\in\hat{G}^n} \hat{1_S}(\chi)^3 \chi(f)
  &= n \sum_{m=0}^{2\epsilon n} \sum_{\substack{m\text{-sparse}~\chi\\H(\chi)\leq \epsilon}} \hat{1_S}(\chi)^3 \chi(f) + \sum_{H(\chi)>\epsilon} \hat{1_S}(\chi)^3 \chi(f)\nonumber\\
  &= n \sum_{m=0}^M \sum_{m\text{-sparse}~\chi} \hat{1_S}(\chi)^3 \chi(f) + O\(n \sum_{m=M+1}^{2\epsilon n} \sum_{m\text{-sparse}~\chi} |\hat{1_S}(\chi)|^3 \)\nonumber\\
  &\qquad\qquad + O\(\sum_{H(\chi)>\epsilon} |\hat{1_S}(\chi)|^3\).\nonumber
\end{align}
Here $M$ is a parameter satisfying $1\ll M \ll \eps n$. By analogy with the circle method from analytic number theory, we refer to the terms of the first sum as \emph{major arcs}, and everything else as \emph{minor arcs}.

In \cite{EMM} we proved a number of bounds for $\hat{1_S}(\chi)$, or for sums of $\hat{1_S}(\chi)$, which allowed us to prove satisfactory estimates for these terms in the case of $G = \Z/n\Z$ with $n$ odd. We made clear our position that there are only notational modifications when $G$ is an arbitrary abelian group of odd order, but there were one or two places where we really did need the full strength of the odd order hypothesis. We state the main results again here, with clearly defined hypotheses. (Although we are stating these results with more general hypotheses than in \cite{EMM}, we need not give new proofs: they are the same.)

\begin{theorem}[Major arcs estimate, Theorem~3.1 from \cite{EMM}]\label{major-arcs}
Let $G$ be an abelian group of order $n$. If $m$ is even then
\[
  \sum_{m\text{-sparse}~\chi} \hat{1_S}(\chi)^3 = \frac{(-1)^{m/2}}{2^{m/2}(m/2)!} \(\frac{n!}{n^{n}}\)^3 + O_m\(\frac1n \(\frac{n!}{n^{n}}\)^3\),
\]
while if $m$ is odd then
\[
  \sum_{m\text{-sparse}~\chi} \hat{1_S}(\chi)^3 = O_m\(\frac1n \(\frac{n!}{n^{n}}\)^3\).
\]
\end{theorem}

\begin{theorem}[$L^\infty$ bound for sparse characters, Proposition~6.1 from \cite{EMM}]\label{L-infty-bound}
Let $G$ be an abelian group of order $n$, where $n$ is odd. If $m\leq n/3$ then
\[
  \max_{m\text{-sparse}~\chi} |\hat{1_S}(\chi)| \leq e^{O(m^{3/2}/n^{1/2} + m^{1/2})} 2^{-m/2} \binom{n}{m}^{-1/2} \frac{n!}{n^n}.
\]
\end{theorem}

\begin{theorem}[Sparseval, Theorem~5.1 from \cite{EMM}]\label{sparseval}
Let $G$ be an abelian group of order $n$. If $m\leq n/2$ then
\[
  \sum_{m\text{-sparse}~\chi} |\hat{1_S}(\chi)|^2 \leq O(m^{1/4}) e^{O(m^{3/2}/n^{1/2})} \binom{n}{m}^{1/2} \(\frac{n!}{n^{n}}\)^2.
\]
\end{theorem}

\begin{theorem}[Square-root cancellation, Theorem~4.1 from \cite{EMM}]\label{square-root-cancellation}
Let $G$ be an abelian group of order $n$. If $\chi = (\psi_1^{a_1}, \dots, \psi_k^{a_k})$, then
\[
  |\hat{1_S}(\chi)| \leq \binom{n+k-1}{k-1}^{1/2} \binom{n}{a_1,\dots,a_k}^{-1/2} \frac{n!}{n^n}.
\]
\end{theorem}


By combining Theorems \ref{major-arcs} to \ref{square-root-cancellation} appropriately we proved, for odd-order $G$ and sufficiently small but constant $\eps>0$,
\[
  \sum_{m=0}^{2\epsilon n} \sum_{\substack{m\text{-sparse}~\chi\\H(\chi)\leq \epsilon}} \hat{1_S}(\chi)^3 = (e^{-1/2} + o(1)) \(\frac{n!}{n^{n}}\)^3
\]
and
\begin{equation}\label{high-entropy bound}
  \sum_{H(\chi)>\epsilon} |\hat{1_S}(\chi)|^3 = O_\epsilon\(\exp(-c\epsilon n) \(\frac{n!}{n^{n}}\)^3\),
\end{equation}
and this proves Theorem~\ref{main-EMM}.

In this paper we follow the same broad strategy, but the following amendments are necessary.
\begin{enumerate}[1.]
  \item We need a ``twisted'' version of the major arcs estimate (Theorem~\ref{major-arcs}) which includes the factor $\chi(f)$. This is a straightforward modification, but it requires rehashing the proof of that estimate. This is where the singular series $\S(f)$ series comes from. See Section~\ref{sec:major-arcs}.
  \item We need a replacement for the $L^\infty$ bound (Theorem~\ref{L-infty-bound}) which does not require the odd-order hypothesis. Our replacement will however only be effective for $m$ up to roughly $n/(\log n)^2$. This is the most technical part of the paper. See Section~\ref{sec:low-entropy}.
  \item Because our replacement for Theorem~\ref{L-infty-bound} is only effective for $m$ up to $n/(\log n)^2$ we need to revisit the proof of \eqref{high-entropy bound} in order to cover a wider entropy range. See Section~\ref{sec:high-entropy}.
\end{enumerate}
Having done all this we will then be in a position to prove our main theorems in Section~\ref{sec:main-theorem}. The two applications mentioned in the introduction are covered in Sections~\ref{sec:latin-cubes} and~\ref{sec:cryptography} respectively.

\section{Major arcs}
\label{sec:major-arcs}
\newcommand{\cP}{\mathcal{P}}

In this section we prove the following theorem, which generalizes and replaces Theorem~\ref{major-arcs}.

\begin{theorem}[Major arcs estimate]\label{major-arcs-replacement}
Let $G$ be a group of order $n$. If $m$ is even then
\begin{align}
  \sum_{m\text{-sparse}~\chi} \hat{1_S}(\chi)^3 \chi(f)
  &= \frac1{(m/2)!} \(-\frac1{2n^2} \sum_{x\in G} |f^{-1}(x)|^2 \)^{m/2} \(\frac{n!}{n^{n}}\)^3\nonumber\\
  &\qquad\qquad + O_m\(\frac1n \(\frac{n!}{n^{n}}\)^3\),\nonumber 
\end{align}
while if $m$ is odd then
\begin{equation}\nonumber 
  \sum_{m\text{-sparse}~\chi} |\hat{1_S}(\chi)|^3 = O_m\(\frac1n \(\frac{n!}{n^{n}}\)^3\).
\end{equation}
\end{theorem}

To prove this we recall some formulae from \cite[Section~3]{EMM}. Given an $m$-sparse character
\[
  \chi = (\chi_1,\dots,\chi_m,0^{n-m}) \in \hat{G}^n,
\]
we say that a partition $\cP$ of $\{1,\dots,m\}$ \emph{kills} $(\chi_1,\dots,\chi_m)$ if $\sum_{i\in P} \chi_i = 0$ for each part $P\in\cP$. Since we are assuming that each $\chi_i$ is nonzero, a killing partition of $(\chi_1,\dots,\chi_m)$ can have at most $\floor{m/2}$ parts. Moreover we have the following two calculations (see (3.3) and (3.4) from \cite{EMM}):
\begin{itemize}
  \item Suppose every killing partition of $(\chi_1,\dots,\chi_m)$ has at most $k$ parts. Then
  \begin{equation}
  \label{k-parts}
    |\hat{1_S}(\chi_1,\dots,\chi_m,0^{n-m})| = O_m\(\frac1{n^{m-k}} \frac{n!}{n^n}\).
  \end{equation}
  \item Suppose that $m$ is even and that $(\chi_1,\dots,\chi_m)$ is killed by a unique partition with $m/2$ parts. Then
  \begin{equation}
  \label{m/2-parts}
    \hat{1_S}(\chi_1,\dots,\chi_m,0^{n-m}) = \frac{(-1)^{m/2}}{n^{m/2}} \frac{n!}{n^n} + O_m\(\frac1{n^{m/2+1}} \frac{n!}{n^n}\).
  \end{equation}
\end{itemize}

Our first claim is that we may ignore all $\chi$ except those whose nonzero coordinates are killed by a unique partition with $m/2$ parts, i.e., exactly those $\chi$ covered by \eqref{m/2-parts}. Note this already proves the odd case of Theorem~\ref{major-arcs-replacement}.

\begin{lemma}\label{lem:Xm}
Let $\mathfrak{M}_m\subset \hat{G}^n$ be the set of all $m$-sparse $\chi$ whose nonzero coordinates are killed by a unique partition with $m/2$ parts. Then
\[
  \sum_{m\text{-sparse}~\chi\notin \mathfrak{M}_m} |\hat{1_S}(\chi)|^3 = O_m\(\frac1n \(\frac{n!}{n^n}\)^3\).
\]
\end{lemma}

\begin{proof}
By permutation-invariance we have
\[
  \sum_{m\text{-sparse}~\chi\notin \mathfrak{M}_m} |\hat{1_S}(\chi)|^3 = \binom{n}{m} \sum_{\substack{\chi_1,\dots,\chi_m\neq 0\\ (\chi_1,\dots,\chi_m,0^{n-m})\notin \mathfrak{M}_m}} |\hat{1_S}(\chi_1,\dots,\chi_m,0,\dots,0)|^3.
\]
We now divide the sum up acoording to the partition $\cP$ of maximal size which kills $(\chi_1,\dots,\chi_m)$ (if there are ties, just choose one). Since there are only $O_m(1)$ such partitions, we may consider each partition individually. Fix such a partition $\cP$ of $\{1,\dots,m\}$, and suppose $\cP$ has $k$ parts.

First consider the case $k < m/2$. The number of $(\chi_1,\dots,\chi_m)$ killed by $\cP$ is $O_m(n^{m-k})$, so by \eqref{k-parts} the contribution to the sum from $\cP$ is
\[
  \binom{n}{m} O_m(n^{m-k}) O_m\(\frac1{n^{m-k}} \frac{n!}{n^n}\)^3 = O_m\(\frac1{n^{m-2k}} \(\frac{n!}{n^n}\)^3\) = O_m\(\frac1n \(\frac{n!}{n^n}\)^3\).
\]

Now consider the case $k = m/2$. By definition of $\mathfrak{M}_m$, every $\chi\notin \mathfrak{M}_m$ killed by $\cP$ is also killed by some other partition with $m/2$ parts. Since the number of partitions is $O_m(1)$, the number of such $\chi$ is $O_m(n^{m/2-1})$, so again we get a satisfactory bound for the contribution from $\cP$:
\[
  \binom{n}{m} O_m(n^{m/2-1}) O_m\(\frac1{n^{m/2}} \frac{n!}{n^n}\)^3 = O_m\(\frac1n \(\frac{n!}{n^n}\)^3\).\qedhere
\]
\end{proof}

It remains to estimate the sum
\[
  \sum_{\chi\in \mathfrak{M}_m} \hat{1_S}(\chi)^3 \chi(f).
\]
Applying \eqref{m/2-parts}, we have
\begin{align*}
  \sum_{\chi\in \mathfrak{M}_m} \hat{1_S}(\chi)^3 \chi(f)
  &= \sum_{\chi\in \mathfrak{M}_m} \(\frac{(-1)^{m/2}}{n^{m/2}} + O_m\(\frac1{n^{m/2+1}}\)\)^3 \(\frac{n!}{n^n}\)^3 \chi(f) \\
  &= \frac{(-1)^{m/2}}{n^{3m/2}} \(\sum_{\chi\in \mathfrak{M}_m} \chi(f)\) \(\frac{n!}{n^n}\)^3 + O_m\(\frac{|\mathfrak{M}_m|}{n^{3m/2+1}} \(\frac{n!}{n^n}\)^3\).
\end{align*}
Since $\mathfrak{M}_m$ has size $\binom{n}{m} O_m(n^{m/2}) = O_m(n^{3m/2})$, this is
\begin{equation}\label{eqn:Mmsum}
  \frac{(-1)^{m/2}}{n^{3m/2}} \(\sum_{\chi\in \mathfrak{M}_m} \chi(f)\) \(\frac{n!}{n^n}\)^3 + O_m\(\frac1n \(\frac{n!}{n^n}\)^3\),
\end{equation}
so we may now concentrate on the sum
\[
  \sum_{\chi \in \mathfrak{M}_m} \chi(f).
\]

We can write
\[
  \sum_{\chi \in \mathfrak{M}_m} \chi(f) = \sum_{|N| = m} \sum_\cP \sum_\chi \chi(f),
\]
where the first sum runs over all subsets $N\subset\{1,\dots,n\}$ of size $m$, the second sums runs over all partitions $\cP$ of $N$ into $m/2$ pairs, and the third sum runs over all $\chi$ with nonzero coordinates exactly on $N$ and killed by the partition $\cP$ and no other. For illustration consider the case $N = \{1,\dots,m\}$ and
\[
  \cP = \{\{1,2\},\dots,\{m-1,m\}\}.
\]
Then the inside sum runs over all choices of
\[
  \chi = (\chi_1,\dots,\chi_m,0^{n-m})
\]
such that $\chi_1,\dots,\chi_m$ are nonzero, $\chi_{2i-1} + \chi_{2i} = 0$ for $i = 1, \dots, m/2$, and moreover such that $\chi_k + \chi_l \neq 0$ unless $\{k,l\} = \{2i-1,2i\}$ for some $i$. However, at the cost of an error of size only $O_m(n^{m/2-1})$ we can just consider the sum over all $\chi_1,\dots,\chi_m$ such that $\chi_{2i-1} + \chi_{2i} = 0$ for each $i$, and because we then have complete character sums we get that
\[
  \sum_\chi \chi(f) = n^{m/2} 1_{f(1) = f(2), \dots, f(m-1) = f(m)} + O_m(n^{m/2-1}).
\]
Similarly, for general $N$ and $\cP$ we have
\begin{equation}
  \sum_\chi \chi(f) = n^{m/2} 1_{f~\text{constant on each}~P\in\cP} + O_m(n^{m/2-1}).\label{eqn:Mmsum-chi-only}
\end{equation}
Now we must sum over $N$ and $\cP$, but by again accepting a negligible error we may simply sum over all ways of choosing $m/2$ pairs from $\{1,\dots,n\}$ \emph{with} replacement. Thus
\begin{align}
  \sum_{|N|=m} \sum_\cP &1_{f~\text{constant on each}~P\in\cP}\nonumber\\
  &= \frac1{(m/2)!2^{m/2}} \sum_{i_1,i_2,\dots,i_{m-1},i_m} 1_{f(i_1) = f(i_2), \dots, f(i_{m-1}) = f(i_m)} + O_m(n^{3m/2-1})\nonumber\\
  &= \frac1{(m/2)!2^{m/2}} \( \sum_{i,j} 1_{f(i)=f(j)}\)^{m/2} + O_m(n^{3m/2-1})\nonumber\\
  &= \frac1{(m/2)!2^{m/2}} \(\sum_{x\in G} |f^{-1}(x)|^2\)^{m/2} + O_m(n^{3m/2-1}).\label{eqn:NPsum}
\end{align}

Thus by combining \eqref{eqn:Mmsum}, \eqref{eqn:Mmsum-chi-only}, and \eqref{eqn:NPsum} we have
\begin{align*}
  \sum_{\chi\in \mathfrak{M}_m} \hat{1_S}(\chi)^3 \chi(f)
  &= \frac1{(m/2)!} \(-\frac1{n^2} \sum_{x\in G} |f^{-1}(x)|^2\)^{m/2} \(\frac{n!}{n^n}\)^3\\
  &\qquad\qquad + O_m\(\frac1n \(\frac{n!}{n^n}\)^3\),
\end{align*}
and this combined with Lemma~\ref{lem:Xm} proves Theorem~\ref{major-arcs-replacement}.

\section{Low-entropy minor arcs}
\label{sec:low-entropy}

In this section we establish a suitable replacement for Theorem~\ref{L-infty-bound} in the even-order case. Our only tool will be the following recursive formula for $\hat{1_S}(\chi)$, which was also our only tool for proving Theorem~\ref{L-infty-bound}. Let $\chi \in \hat{G}^n$ be a character with exactly $m$ nonzero coordinates, say
\[
  \chi = (\chi_1,\dots,\chi_m,0^{n-m}),
\]
where each $\chi_i$ is nonzero. Then
\begin{align}
  \hat{1_S}(\chi)
  &= \frac{(n-m)!}{n^n} \sump_{x_1,\dots,x_{m-1}} \sum_{x_m\neq x_1,\dots,x_{m-1}} \chi_1(x_1) \cdots \chi_m(x_m)\nonumber\\
  &= - \frac{(n-m)!}{n^n} \sump_{x_1,\dots,x_{m-1}} \sum_{x_m = x_1,\dots,x_{m-1}} \chi_1(x_1) \cdots \chi_m(x_m)\nonumber\\
  &= - \frac1{n-m+1} \sum_{i=1}^{m-1} \hat{1_S}(\chi^i),\label{recursion}
\end{align}
where here we write $\chi^i$ for the character
\[
  \chi^i = (\chi_1, \dots, \chi_i + \chi_m, \dots, \chi_{m-1}, 0^{n-m+1}).
\]
Note that each $\chi^i$ has either $m-1$ or $m-2$ nonzero coordinates (depending on whether $\chi_i = - \chi_m$), so repeated application of \eqref{recursion} (and $\hat{1_S}(0) = n!/n^n$) constitutes a method of computing $\hat{1_S}(\chi)$. The more relevant thing for us however is the bound implied by \eqref{recursion}:
\begin{equation}\label{recursive-bound}
  |\hat{1_S}(\chi)| \leq \frac{1}{n-m+1} \sum_{i=1}^{m-1} |\hat{1_S}(\chi^i)|.
\end{equation}
An immediate consequence of this bound is the following lemma, which, although weaker than Theorem~\ref{L-infty-bound} by a factor of roughly $2^{m/2}$, does not rely on the absence of $2$-torsion.

\begin{lemma}\label{elementary-bound}
If $\chi \in \hat{G}^n$ has exactly $m$ nonzero coordinates, where $m \leq n/2$, then
\[
  |\hat{1_S}(\chi)| \leq \binom{n}{m}^{-1/2} \frac{n!}{n^n}.
\]
\end{lemma}

\begin{proof}
The claim is true when $m = 0$ (with equality) and when $m=1$ (because $\hat{1_S}(\chi) = 0$), so assume $m > 1$. By \eqref{recursive-bound},
\[
  |\hat{1_S}(\chi)| \leq \frac{m-1}{n-m+1} \max_i |\hat{1_S}(\chi^i)|.
\]
Since each $\chi^i$ has either $m-1$ or $m-2$ nonzero coordinates, by induction (and $m\leq n/2$) we have
\begin{align*}
  |\hat{1_S}(\chi)|
  &\leq \frac{m-1}{n-m+1} \binom{n}{m-2}^{-1/2} \frac{n!}{n^n}\\
  &\leq \(\frac{m(m-1)}{(n-m)(n-m+1)}\)^{1/2} \binom{n}{m-2}^{-1/2} \frac{n!}{n^n}\\
  &= \binom{n}{m}^{-1/2} \frac{n!}{n^n}.\qedhere
\end{align*}
\end{proof}

Unfortunately, there are many approximate equality cases of Lemma~\ref{elementary-bound}. For example if
\[
  \chi = (\chi_0^m, 0^{n-m}),
\]
where $2\chi_0 = 0$,  then
\[
  |\hat{1_S}(\chi)| = \frac{(m-1)(m-3) \cdots 1}{(n-m+1)(n-m+3) \cdots n} \frac{n!}{n^n} \approx \binom{n}{m}^{-1/2} \frac{n!}{n^n}.
\]
This follows from \eqref{recursion}. Moreover we also expect approximate equality whenever $\chi$ approximately has this form, say whenever
\[
  \chi = (\chi_0^{m-k}, \chi_1, \dots, \chi_k, 0^{n-m})
\]
with $2\chi_0 = 0$ and $k$ much smaller than $m$. However, our next theorem asserts that we have an exponential improvement to Lemma~\ref{elementary-bound} whenever $\chi$ does not have essentially this form. This is our replacement for Theorem~\ref{L-infty-bound}.

\begin{theorem}\label{new-theorem}
Let $\chi \in \hat{G}^n$ be a character with exactly $m$ nonzero coordinates, where $m \leq n/2$, and such that no nonzero $2$-torsion coordinate is repeated more than $(1-\delta)m$ times, where $0 \leq \delta \leq 1/2$. Then
\[
  |\hat{1_S}(\chi)| \leq e^{O(m^{3/2}/n^{1/2} + m^{1/2})} (1-\delta)^{m/2} \binom{n}{m}^{-1/2} \frac{n!}{n^n}.
\]
\end{theorem}
\begin{proof}
Fixing $\delta$, let $X_m$ be the set of $\chi\in\hat{G}^n$ with exactly $m$ nonzero coordinates and such that no nonzero $2$-torsion coordinate is repeated more than $(1-\delta)m + 3$ times. This set $X_m$ includes all the $\chi$ covered by the theorem. (We allow the extra $3$ for technical reasons to do with the induction.) Let
\[
  U_m = \max_{\chi \in X_m} |\hat{1_S}(\chi)|.
\]
We seek a bound for $U_m$.

Let $\chi\in X_m$. We may assume by permuting coordinates that
\[
  \chi = (\chi_1,\dots,\chi_m,0^{n-m}).
\]
We also assume, by permuting coordinates if necesary, the following: if $\chi$ has a majority coordinate (i.e., if more than half of the nonzero coordinates $\chi_1, \dots, \chi_m$ are equal to some particular $\chi_0$) then $\chi_m$ is the majority coodinate.

Now suppose we apply \eqref{recursive-bound}. Consider the resulting characters $\chi^i$:
\[
  \chi^i = (\chi_1, \dots, \chi_i + \chi_m, \dots, \chi_{m-1}, 0^{n-m+1}).
\]
We claim that $\chi^i \in X_{m-1} \cup X_{m-2}$. To see this, note first that coordinates other than $\chi_m$ are repeated in $\chi^i$ at most
\[
  m/2 + 1 \leq (1-\delta)(m-2) + 3
\]
times, so we need only worry about $\chi_m$. If $2\chi_m \neq 0$ then there is nothing to worry about, so assume $2\chi_m = 0$. Suppose $\chi_m$ repeats $k$ times in $\chi$. Then depending on whether $\chi_i = \chi_m$, $\chi^i$ has either $m-1$ nonzero coordinates and $k-1$ repetitions of $\chi_m$, or $m-2$ nonzero coordinates and $k-2$ repetitions of $\chi_m$. Both of these possibilities are fine: since $k \leq (1-\delta)m + 3$ we have
\begin{align*}
  k - 1 & \leq (1-\delta)(m-1) + 3,\\
  k - 2 &\leq (1-\delta)(m-2) + 3.
\end{align*}
Thus indeed either $\chi^i \in X_{m-1}$ or $\chi^i \in X_{m-2}$.

For how many $i$ can we have $\chi^i \in X_{m-2}$? Recall that $\chi^i \in X_{m-2}$ if and only if $\chi_i = - \chi_m$. If more than half of the indices $i$ satisfy $\chi_i = - \chi_m$ then $-\chi_m$ is a majority coordinate. By arrangement, this can only happen if $-\chi_m = \chi_m$, i.e., if $\chi_m$ is $2$-torsion. By definition of $X_m$, then, $\chi_m$ is repeated at most $(1-\delta)m + 3$ times. Thus there can be at most $(1-\delta)m + 2$ indices $i$ for which $\chi^i \in X_{m-2}$.

Putting this all together with \eqref{recursive-bound}, we have
\begin{equation}\label{Um-recurrence}
  U_m \leq \frac{m-1}{n-m+1} \max\(U_{m-1},\frac{\delta m - 3}{m-1} U_{m-1} + \frac{(1-\delta) m + 2}{m-1} U_{m-2}\)
\end{equation}
for all $m \geq 2$. This recurrence is very similar to the one we analyzed in \cite[Section~6]{EMM}. We use the same method here. 

In \eqref{Um-recurrence}, the $\max$ covers the possibility that $U_{m-1} > U_{m-2}$, but typically we expect $U_{m-1}$ to be much smaller than $U_{m-2}$, so in fact only the term involving $U_{m-2}$ should be important. To take advantage of this heuristic, we renormalize. Write
\[
  \alpha_m = \frac{(1-\delta)m+2}{n-m+1},\quad
  \beta_m = \frac{\delta m-3}{n-m+1},\quad
  \gamma_m = \frac{m-1}{n-m+1},
\]
and define
\[
  V_m=(\alpha_1\alpha_2\dots \alpha_m)^{-1/2} U_m.
\]
Then \eqref{Um-recurrence} becomes
\begin{equation}\label{Vm-recurrence}
  V_m \leq \max \(\gamma_m\alpha_m^{-1/2}\, V_{m-1},\  \alpha_m^{1/2}\alpha_{m-1}^{-1/2}\, V_{m-2} + \beta_m\alpha_{m}^{-1/2}\, V_{m-1} \).
\end{equation}
If we now define
\begin{align*}
  v_m &= \(\begin{array}{c}V_m\\ V_{m-1}\end{array}\),\\
  M_m &= \(\begin{array}{cc}\gamma_m \alpha_m^{-1/2}&0\\1&0\end{array}\),\\
  N_m &= \(\begin{array}{cc}\beta_m \alpha_m^{-1/2} & \alpha_m^{1/2} \alpha_{m-1}^{-1/2} \\ 1 & 0 \end{array}\),
\end{align*}
then we can write \eqref{Vm-recurrence} equivalently as
\[
  \|v_m\|_2 \leq \max\(\|M_m v_{m-1}\|_2, \|N_mv_{m-1}\|_2 \).
\]
To finish we bound the operator norms
\[
  \|M_m\|_{2\to2},\quad\|N_m\|_{2\to2},
\]
and then use the bound
\begin{align}
  |V_m|
  &\leq \|v_m\|_2\nonumber\\
  &\leq \prod_{r=2}^m \max\(\|M_m\|_{2\to 2}, \|N_m\|_{2\to 2}\) \|v_1\|_2\nonumber\\
  &= \prod_{r=2}^m \max\(\|M_m\|_{2\to 2}, \|N_m\|_{2\to 2}\) \frac{n!}{n^n}.\label{operatornormbound}
\end{align}
Since $\alpha_m\alpha_{m-1}^{-1} = 1+O(1/m)$ and $\beta_m^2\alpha_m^{-1} = O(m/n)$, we have
\[
  \tr\left(N_m^T N_m\right) = 2 + O(m/n + 1/m)
\]
and
\[
  \det\left(N_m^T N_m\right) = 1 + O(1/m).
\]
Thus (because $\|N_m\|_{2\to 2}^2$ is equal to the larger eigenvalue of $N_m^T N_m$) we have
\[
  \|N_m\|_{2\to 2} = 1 + O(m^{1/2}/n^{1/2} + 1/m^{1/2}).
\]
Similarly, because $\gamma_m^2\alpha_m^{-1} = O(m/n)$, we have
\[
  \|M_m\|_{2\to 2} = 1 + O(m/n).
\]
Thus from \eqref{operatornormbound} we have
\begin{align*}
  |V_m|
  &\leq \prod_{r=2}^m \(1 + O(r^{1/2}/n^{1/2} + 1/r^{1/2})\) \frac{n!}{n^n}\\
  &= e^{O(m^{3/2}/n^{1/2} + m^{1/2})} \frac{n!}{n^n}.
\end{align*}
Since
\[
  \alpha_1 \cdots \alpha_m = \prod_{r=1}^m \( (1 + O(1/r))\frac{(1-\delta)r}{n-r+1}\) = e^{O(\log m)} (1-\delta)^m \binom{n}{m}^{-1},
\]
the theorem follows.
\end{proof}

\begin{corollary}\label{low-entropy}
Assume $0\leq \delta\leq 1/2$. For $m\leq n/2$ we have
\begin{align*}
  \sum_{m\text{-sparse}~\chi} |\hat{1_S}(\chi)|^3
  &\leq e^{O(m^{3/2}/n^{1/2} + m^{1/2})} (1-\delta)^{m/2} \(\frac{n!}{n^n}\)^3\\
  &\qquad + 2^m n^{\delta m + 1} \binom{n}{m}^{-1/2} \(\frac{n!}{n^n}\)^3.
\end{align*}
Thus, for some constants $\eta < 1$ and $M_0$, for all $M\geq M_0$ we have
\[
  \sum_{m=M}^{n/(\log n)^2} \sum_{m\text{-sparse}~\chi} |\hat{1_S}(\chi)|^3 \leq O\(\eta^M \(\frac{n!}{n^n}\)^3\) + O\( e^{-cn^{1/2}} \(\frac{n!}{n^n}\)^3\).
\]
\end{corollary}
\begin{proof}
Again denote by $X_m$ the set of characters to which Theorem~\ref{new-theorem} applies. By combining Theorem~\ref{new-theorem} with Theorem~\ref{sparseval} we have
\begin{align*}
  \sum_{\chi\in X_m} |\hat{1_S}(\chi)|^3
  &\leq \max_{\chi\in X_m} |\hat{1_S}(\chi)| \sum_{\chi\in X_m} |\hat{1_S}(\chi)|^2\\
  &\leq e^{O(m^{3/2}/n^{1/2} + m^{1/2})} (1 - \delta)^{m/2} \(\frac{n!}{n^n}\)^3.
\end{align*}
This is the first term in the claim. On the other hand the number of $m$-sparse $\chi$ not in $X_m$ is at most
\[
  \binom{n}{m} 2^m n^{\delta m + 1}
\]
(you choose $m$ coordinates to be nonzero, some subset of at most $\delta m$ of these to be different from $\chi_0$, and finally the values of the coordinates). Thus by Lemma~\ref{elementary-bound} we have
\[
  \sum_{\substack{m\text{-sparse}~\chi\\\chi\notin X_m}} |\hat{1_S}(\chi)|^3 \leq 2^m n^{\delta m+1} \binom{n}{m}^{-1/2} \(\frac{n!}{n^n}\)^3.
\]

To prove the second part of the corollary, consider the second term first. Using the simple bound $\binom{n}{m} \geq (n/m)^m$, we have
\begin{align*}
  \log\(2^m n^{\delta m+1} \binom{n}{m}^{-1/2}\)
  &\leq \log\(2^m n^{\delta m+1} (n/m)^{-m/2}\)\\
  &= O(m + \log n) + \delta m \log n - (m/2) \log(n/m).
\end{align*}
As long as $n/m$ and $m$ are both larger than a sufficiently large constant (depending on the constant implicit in the $O$ above), then $m + \log n$ is negligible compared to $(m/2) \log (n/m)$, so if we put
\[
  \delta = c \log(n/m)/\log n
\]
for a sufficiently small constant $c$ then we have
\[
  \log\(2^m n^{\delta m+1} \binom{n}{m}^{-1/2}\) \leq -(m/4) \log (n/m).
\]

Now consider the first term. we have
\begin{align*}
  \log\(e^{O(m^{3/2}/n^{1/2} + m^{1/2})} (1-\delta)^{m/2}\)
  \leq O(m^{3/2}/n^{1/2} + m^{1/2}) - \delta m/2\\
  \qquad\qquad= \(-(c/2) \frac{\log(n/m)}{\log n} + O((m/n)^{1/2} + m^{-1/2})\) m.
\end{align*}
As long as $1\leq m = o(n (\log\log n)^2/(\log n)^2)$ it is easy to see that
\[
  (m/n)^{1/2} = o\( \frac{\log(n/m)}{\log n}\):
\]
indeed, this is equivalent to
\[
  \log n = o((n/m)^{1/2} \log (n/m)).
\]
It is also easy to see that
\[
  m^{-1/2} = o\( \frac{\log(n/m)}{\log n}\)
\]
as long as $m\to\infty$ (consider the cases $m\leq n^{1/2}$ and $m\geq n^{1/2}$ separately). Thus we have 
\[
  \log\(e^{O(m^{3/2}/n^{1/2} + m^{1/2})} (1-\delta)^{m/2}\)
  \leq -(c/4) \frac{\log(n/m)}{\log n} m
\]
as long as $m\geq M_0$ for some constant $M_0 = M_0(c)$.

Putting these two bounds together, as long as $M_0\leq m \leq n/(\log n)^2$ we have
\[
  \sum_{m\text{-sparse}~\chi} |\hat{1_S}(\chi)|^3 \leq \exp\(-c' m \frac{\log(n/m)}{\log n}\) \(\frac{n!}{n^n}\)^3.
\]
In particular if $m \geq n^{0.6}$ we have
\[
  \sum_{m\text{-sparse}~\chi} |\hat{1_S}(\chi)|^3 \leq \exp\(-c' n^{0.6}/\log n\) \(\frac{n!}{n^n}\)^3.
\]
In this range we can afford to simply take the maximum over $m$ and accept an additional factor of $n$. On the other hand if $m \leq n^{0.6}$ we have
\[
  \sum_{m\text{-sparse}~\chi} |\hat{1_S}(\chi)|^3 \leq \exp\( - 0.6 c' m\) \(\frac{n!}{n^n}\)^3.
\]
The sum of this expression over $m\geq M$ is a geometric series and thus dominated by the $m = M$ term. This proves the corollary.
\end{proof}

\section{High-entropy minor arcs}
\label{sec:high-entropy}

In this section we extend the proof of \eqref{high-entropy bound} to cover a wider entropy range. This mainly requires more careful counting. The methods of this section are effective for characters with entropy as low as $n^{-1/2 + \epsilon}$, but for simplicity we only consider entropy as low as $(\log n)^{-100}$.

We first need a lemma relating our definition of entropy to the usual one.

\begin{lemma}\label{entropy}
Let $H = H(\chi) = \frac1n \log \binom{n}{a_1,\dots,a_n}$. If $H \geq (\log n)^{-100}$ then
\[
  H = \(1 + O\(\frac{\log\log n}{\log n}\)\) \sum_{i=1}^n \frac{a_i}{n} \log \frac{n}{a_i}.
\]
\end{lemma}
\begin{proof}
We use Stirling's formula in the form
\[
  \log m! = m \log m - m + O(\log (m+2)) \qquad (m\geq 0).
\]
This implies
\begin{equation}\label{entropy-eqn}
  \frac1n \log\binom{n}{a_1,\dots,a_n}
  = \sum_{a_i > 0}\(\frac{a_i}{n}\log \frac{n}{a_i} + O\(\frac{\log (a_i + 2)}n\)\) + O\(\frac{\log n}{n}\).
\end{equation}
Suppose $k$ of the $a_i$ are nonzero. Note first of all that \eqref{entropy-eqn} implies a bound on $k$: we have
\begin{align*}
  H
  &= \frac1n \log \binom{n}{a_1,\dots,a_n}\\
  &\geq \frac1n \log \(\frac{n!}{(n-k+1)! 1^{k-1}}\)\\
  &= \frac{n-k+1}n \log\frac{n}{n-k+1} + \frac{k-1}n\log n + O\(\frac{k}n + \frac{\log n}n\)\\
  &= \frac{k}n \log n + O\( \frac{k}n + \frac{\log n}n\).
\end{align*}
This implies $k \leq O(Hn / \log n)$. Thus by concavity of $\log$ we have
\begin{align*}
  \frac1n \sum_{a_i>0} \log(a_i+2)
  &\leq \frac{k}{n} \log\(\frac1k \sum_{a_i>0} (a_i+2)\)\\
  &= \frac{k}{n} \log \(\frac{n}{k} + 2\)\\
  &\leq O\(\frac{H}{\log n} \( \log \log n + \log H^{-1}\)\)\\
  &= O\(\frac{H \log \log n}{\log n}\).
\end{align*}
In the last line we used the assumption $H \geq (\log n)^{-100}$. Inserting this into \eqref{entropy-eqn} gives
\[
  H = \sum_{i=1}^n \frac{a_i}n \log \frac{n}{a_i} + O\( \frac{H\log\log n}{\log n}\),
\]
which is equivalent to the claim.
\end{proof}

\begin{lemma}\label{orbit-count}
Suppose $H \geq (\log n)^{-100}$. Then the number of characters $\chi\in\hat{G}^n$ of entropy at most $H$ is bounded by
\[
  \exp\(Hn + O\(\frac{H n \log \log n}{\log n}\)\).
\]
\end{lemma}

The important part of the conclusion is that the bound has the form $e^{Hn + o(Hn)}$.

\begin{proof}
If $\chi$ has entropy $H$ then the orbit of $\chi$ under permutation of coordinates has size $e^{Hn}$, so it suffices to show that the number of orbits is at most
\[
  \exp\( O\(\frac{H n \log \log n}{\log n}\)\).
\]
Choosing an orbit is equivalent to fixing the multiplicities $a_i$ of each coordinate, so we are in the business of counting solutions to
\begin{equation}\label{chi-orbit}
\begin{cases}
a_1 + \cdots + a_n &= n,\\
\frac1n \log \binom{n}{a_1,\dots,a_n} &\leq H.
\end{cases}
\end{equation}

Let $t$ be the sum of the $a_i$ for which $a_i\leq n^{1/2}$, and note that at most $n^{1/2}$ of the $a_i$ are not included in this set. Let
\[
  H' = \sum_{i=1}^n \frac{a_i}n \log \frac{n}{a_i}.
\]
Then
\[
  H' \geq \frac{t}{n} \log n^{1/2} = \frac{t}{2n} \log n.
\]
On the other hand Lemma~\ref{entropy} implies $H' \leq O(H)$. Thus
\[
  t \leq O\(\frac{H n}{\log n}\).
\]

Now we can select a solution to \eqref{chi-orbit} by choosing up to $n^{1/2}$ indices $i$ at which we will have $a_i > n^{1/2}$, choosing the values of these $a_i$, and then choosing the values of all other $a_i$ in such a way that their sum does not exceed $O(H n / \log n)$. Thus the number of solutions to \eqref{chi-orbit} is at most
\begin{align*}
  n^{O(n^{1/2})} \binom{n + O( H n / \log n)}{O(H n / \log n)}
  &= n^{O(n^{1/2})} \exp\(O\( \frac{Hn}{\log n}\log \frac{\log n}{H}\)\)\\
  &= \exp\(O\(\frac{\log\log n}{\log n} + \frac{\log H^{-1}}{\log n}\) Hn\)\\
  &= \exp\(O\(\frac{H n \log \log n}{\log n}\)\).\qedhere
\end{align*}
\end{proof}

\begin{theorem}\label{med-to-high-entropy}
We have the following bound over high-entropy characters:
\[
  \sum_{H(\chi) \geq (\log n)^{-100}} |\hat{1_S}(\chi)|^3 \leq e^{-c n (\log n)^{-100}} \(\frac{n!}{n^n}\)^3.
\]
\end{theorem}
\begin{proof}
Suppose $\chi\in\hat{G}^n$ is a character of entropy at least $(\log n)^{-100}$. By Theorem~\ref{square-root-cancellation} we have
\[
  |\hat{1_S}(\chi)| \leq \binom{n+k-1}{k-1}^{1/2} e^{-Hn/2} \frac{n!}{n^n},
\]
where $k$ is the number of distinct coordinates of $\chi$. We saw in the proof of Lemma~\ref{entropy} that $k \leq O(H n / \log n)$. Thus
\[
  \binom{n+k-1}{k-1} \leq \exp\(O\(\frac{H n \log \log n}{\log n}\)\) = \exp(o(Hn)).
\]
We therefore have
\[
  \sum_{H(\chi) \geq (\log n)^{-100}} |\hat{1_S}(\chi)|^3
  \leq \sum_{H(\chi)\geq (\log n)^{-100}} e^{-3H(\chi)n/2 + o(H(\chi)n)} \(\frac{n!}{n^n}\)^3.
\]
Thus by applying Lemma~\ref{orbit-count} and a dyadic decomposition on $H$ we have
\begin{align*}
  \sum_{H(\chi) \geq (\log n)^{-100}} |\hat{1_S}(\chi)|^3
  &\leq \sum_{\substack{(\log n)^{-100} \leq H \leq \log n\\H~\text{dyadic}}} e^{-Hn/2 + o(Hn)} \(\frac{n!}{n^n}\)^3\\
  &\leq O(\log\log n) e^{-c n (\log n)^{-100}} \(\frac{n!}{n^n}\)^3.\qedhere
\end{align*}
\end{proof}

\section{Proofs of the main theorems}
\label{sec:main-theorem}

We have now assembled all the tools we need to prove our main theorems. We closely follow the outline from Section~\ref{sec:recap}.

\begin{proof}[Proof of Theorem~\ref{main-f}]
As explained in Section~\ref{sec:recap}, the number of solutions to $\pi_1+\pi_2+\pi_3 = f$ is
\[
  n^{2n} \sum_{\chi\in\hat{G}^n} \hat{1_S}(\chi)^3 \chi(f),
\]
and we have the approximation
\begin{align}
  \sum_{\chi\in\hat{G}^n} \hat{1_S}(\chi)^3 \chi(f)
  &= n \sum_{m=0}^M \sum_{m\text{-sparse}~\chi} \hat{1_S}(\chi)^3\chi(f)\nonumber\\
  &\qquad\qquad + O\(n \sum_{m=M+1}^{2\epsilon n} \sum_{m\text{-sparse}~\chi} |\hat{1_S}(\chi)|^3 \)\nonumber\\
  &\qquad\qquad + O\(\sum_{H(\chi)>\epsilon} |\hat{1_S}(\chi)|^3\),\label{decomposition-repeat}
\end{align}
provided that $\sum_{i=1}^n f(i) = \Sigma G$. Put $\epsilon = (\log n)^{-100}$. As long as $M\geq M_0$, Corollary~\ref{low-entropy} implies that the second term of \eqref{decomposition-repeat} has size
\[
  O\( \eta^M n \(\frac{n!}{n^n}\)^3\) + O\( e^{-cn^{1/2}} \(\frac{n!}{n^n}\)^3\),
\]
while Theorem~\ref{med-to-high-entropy} implies that the last term has size
\[
  O\( e^{-cn(\log n)^{-100}} \(\frac{n!}{n^n}\)^3\).
\]
Meanwhile, by Theorem~\ref{major-arcs-replacement} we have
\begin{align*}
  \sum_{m=0}^M \sum_{m\text{-sparse}~\chi} \hat{1_S}(\chi)^3 \chi(f)
  &= \sum_{\substack{m \leq M\\m~\text{even}}} \frac1{(m/2)!} \(-\frac1{2n^2} \sum_{x\in G} |f^{-1}(x)|^2\)^{m/2} \(\frac{n!}{n^n}\)^3\\
  &\qquad\qquad + O_M\(\frac1n \(\frac{n!}{n^n}\)^3\).
\end{align*}
Thus if $M\to\infty$ sufficiently slowly as a function of $n$ then we have
\[
  \sum_{m=0}^M \sum_{m\text{-sparse}~\chi} \hat{1_S}(\chi)^3 \chi(f) = \(\exp\(-\frac1{2n^2} \sum_{x\in G} |f^{-1}(x)|^2\) + o(1)\) \(\frac{n!}{n^n}\)^3.
\]
This proves the theorem.
\end{proof}

The proof of Theorem~\ref{thm:main-4} is similar, but a little easier because we don't need the careful major arc calculations.

\begin{proof}[Proof of Theorem~\ref{thm:main-4}]
The number of solutions to
\[
  \pi_1 + \cdots + \pi_d = f
\]
is
\[
  n^{(d-1)n} 1_S^{*d}(f) = n^{(d-1)n} \sum_{\chi\in\hat{G}^n} \hat{1_S}(\chi)^d \chi(f).
\]
Let $\hat{G}_d\subset \hat{G}^n$ be the set of characters of the form $\chi = (\chi_0^n)$. Then
\[
  \sum_{\chi\in\hat{G}_d} \hat{1_S}(\chi)^d \chi(f) = \sum_{\chi_0\in\hat{G}} \hat{1_S}(\chi_0^n)^d \chi_0(\tx\sum f) = \sum_{\chi_0 \in \hat{G}} \chi_0(d\Sigma G + \Sigma f) \(\frac{n!}{n^n}\)^d.
\]
Assuming $\sum_{i=1}^n f(i) = d \Sigma G$, this is $n(n!/n^n)^d$. 

On the other hand, for $d\geq 4$ we have
\[
  \sum_{\chi\notin\hat{G}_d} |\hat{1_S}(\chi)|^d \leq \max_{\chi\notin\hat{G}_d} |\hat{1_S}(\chi)|^{d-3} \sum_{\chi\in\hat{G}^n} |\hat{1_S}(\chi)|^3.
\]
It follows from Lemma~\ref{elementary-bound} and shift-invariance of $|\hat{1_S}(\chi)|$ (see \eqref{shift-invariance}) that
\[
  \max_{\chi\notin\hat{G}_d} |\hat{1_S}(\chi)| \leq \binom{n}{2}^{-1/2} \frac{n!}{n^n} = O\(\frac1n \(\frac{n!}{n^n}\)\).
\]
Meanwhile we can decompose the sum
\[
  \sum_{\chi\in\hat{G}^n} |\hat{1_S}(\chi)|^3.
\]
as
\begin{align*}
 O\( n \sum_{m=0}^{M} \sum_{m\text{-sparse}~\chi} |\hat{1_S}(\chi)|^3\)
&+ 
O\( n \sum_{m=M+1}^{2\epsilon n} \sum_{m\text{-sparse}~\chi} |\hat{1_S}(\chi)|^3\)\\
&\qquad + O\( \sum_{H(\chi)>\epsilon} |\hat{1_S}(\chi)|^3\).
\end{align*}
The second two terms are bound exactly as in the previous proof (i.e., by the second part of Corollary~\ref{low-entropy} and Theorem~\ref{high-entropy bound}), while the first can be bound using the first part of Corollary~\ref{low-entropy}. The result is that\footnote{In fact a finer analysis using major arc calculations proves that
\[
  \sum_{\chi\in\hat{G}^n} |\hat{1_S}(\chi)|^3 = (e^{1/2} + o(1))n \(\frac{n!}{n^n}\)^3,
\]
but we don't need this.}
\[
  \sum_{\chi\in\hat{G}^n} |\hat{1_S}(\chi)|^3 \leq O\(n \(\frac{n!}{n^n}\)^3\).
\]
Thus
\[
  \max_{\chi\notin\hat{G}_d} |\hat{1_S}(\chi)|^{d-3} \sum_{\chi\in\hat{G}^n} |\hat{1_S}(\chi)|^3 \leq O_d\( n^{4-d} \(\frac{n!}{n^n}\)\),
\]
and the theorem is proved.
\end{proof}

Finally we prove Theorem~\ref{thm:L2}.

\begin{proof}[Proof of Theorem~\ref{thm:L2}]
We have
\begin{equation}\label{nmL2}
  \norm{\(\frac{n^m(n-m)!}{n!}\)^2 1_{S_m} * 1_{S_m} - 1}_2^2 = \(\frac{n^m(n-m)!}{n!}\)^4 \sum_{\substack{\chi\in\hat{G}^m\\\chi\neq0}} |\hat{1_{S_m}}(\chi)|^4.
\end{equation}
The Fourier transform of $1_{S_m}$ is related to that of $1_S$ by
\[
  \hat{1_{S_m}}(\chi_1,\dots,\chi_m)
   = \frac{n^{n-m}}{(n-m)!} \hat{1_S}(\chi_1,\dots,\chi_m,0^{n-m}).
\]
Thus we can expand \eqref{nmL2} as
\begin{equation}\nonumber 
  \(\frac{n^m(n-m)!}{n!}\)^4 \sum_{\substack{\chi\in\hat{G}^m\\ \chi\neq 0}} |\hat{1_{S_m}}(\chi)|^4
  = \(\frac{n^n}{n!}\)^4 \sum_{\substack{\chi\in\hat{G}^m\\ \chi\neq 0}} |\hat{1_S}(\chi_1,\dots,\chi_m,0^{n-m})|^4.
\end{equation}
Now it is easy to see (say by combining Lemma~\ref{elementary-bound} with some basic major arc calculations) that we have
\[
  |\hat{1_S}(\chi)| \leq O\(\frac1{n^3} \(\frac{n!}{n^n}\)\)
\]
unless $\chi$ is, up to translation, at most $4$-sparse. Let $X$ be the set of these $\chi$. Then
\[
  \sum_{\chi\notin X} |\hat{1_S}(\chi)|^4 \leq O\(\frac1{n^4}\(\frac{n!}{n^n}\)\) \sum_\chi |\hat{1_S}(\chi)|^3 \leq O\(\frac1{n^3} \(\frac{n!}{n^n}\)^4\),
\]
since as we saw in the previous proof we have $\sum_\chi |\hat{1_S}(\chi)|^3 \leq O(n(n!/n^n)^3)$. Finally, it is straightforward to calculate that, as long as $m\leq n-1$, we have
\[
  \(\frac{n^n}{n!}\)^4 \sum_{\substack{\chi\in\hat{G}^m\\(\chi,0^{n-m})\in X, \chi\neq 0}} |\hat{1_S}(\chi_1,\dots,\chi_m,0^{n-m})|^4 = O(m^2/n^3),
\]
the dominant contribution coming from $2$-sparse characters, and this proves the claim. 

The very last part of the theorem, about total variation distance, holds because the $L^2$ distance bounds the $L^1$ distance, and the $L^1$ distance is exactly twice the total variation distance.
\end{proof}

\section{Transversals in Latin hypercubes}
\label{sec:latin-cubes}

A \emph{Latin square of order $n$} is an $n \times n$ grid populated with the symbols $\{1,\dots,n\}$ in such a way that no symbol occurs more than once in any given row or column. A \emph{transversal} in a Latin square is a selection of $n$ entries not repeating any row, column, or symbol. See Figure \ref{latin-square} for an example of a transversal in a particular Latin square of order $7$. See Wanless~\cite{wanless} for an extensive survey about transversals in Latin squares.

\begin{figure}[h]
\begin{tabular}{ccccccc}
\circled{1} & 2 & 3 & 4 & 5 & 6 & 7 \\
2 & 6 & 1 & \circled{7} & 3 & 5 & 4 \\
6 & 5 & 7 & 2 & \circled{4} & 1 & 3 \\
5 & 4 & \circled{6} & 3 & 2 & 7 & 1 \\
4 & 7 & 2 & 6 & 1 & 3 & \circled{5} \\
3 & 1 & 4 & 5 & 7 & \circled{2} & 6 \\
7 & \circled{3} & 5 & 1 & 6 & 4 & 2
\end{tabular}
\caption{A Latin square and a transversal}\label{latin-square}
\end{figure}

More generally, a \emph{Latin (hyper)cube of dimension $d$ and order $n$} is a $d$-dimensional grid of side length $n$ filled out with $\{1,\dots,n\}$ in such a way that no symbol occurs more than once in the same (axis-aligned) line. A \emph{transversal} in a $d$-dimensional Latin cube is a selection of $n$ entries not repeating any hyperplane or symbol.

Now let $G$ be an abelian group of order $n$. We claim that counting solutions to $\pi_1 + \cdots + \pi_d = \pi_{d+1}$ in $G$ is equivalent to counting transversals in a certain $d$-dimensional Latin cube. Let $\pi:\{1,\dots,n\}\to G$ be a bijection. The \emph{Latin cube $L^d(G,\pi)$ induced by $G$ and $\pi$} is the $d$-dimensional grid of side length $n$ whose $(i_1,\dots,i_d)$-entry is
\begin{equation}\label{LGpidef}
  x = \pi^{-1}\(\pi(i_1) + \cdots + \pi(i_d)\).
\end{equation}
Clearly this defines a Latin cube: if we alter any one of the $i_j$ while keeping the others fixed then by the group property we must alter $x$. One can think of $L^2(G,\pi)$ as the multiplication (or addition) table of $G$, and likewise $L^d(G,\pi)$ as the graph of the $d$-wise iterated group operation.

\begin{lemma}
Let $S \subset G^n$ be the set of bijections $\pi:\{1,\dots,n\}\to G$. The transversals of $L^d(G,\pi)$ are in bijection with the solutions to
\begin{equation}\label{pi-sum}
  \pi_1 + \cdots + \pi_d = \pi \qquad(\pi_1,\dots,\pi_d\in S).
\end{equation}
In particular $T(L^d(G,\pi)) n!$ is the number of solutions to
\begin{equation}\label{pi-sum-2}
  \pi_1 + \cdots + \pi_d = \pi_{d+1} \qquad(\pi_1,\dots,\pi_{d+1}\in S).
\end{equation}
\end{lemma}
\begin{proof}
Fix a transversal in $L^d(G,\pi)$. For each symbol $x\in\{1,\dots,n\}$ let
\begin{equation}\label{transversalx}
  (\pi^{-1}\pi_1(x),\dots,\pi^{-1}\pi_d(x))
\end{equation}
be the coordinates of the symbol $x$ in the transversal. Since each symbol occurs exactly once in a transversal, $\pi_1,\dots,\pi_d$ are well defined functions $\{1,\dots,n\} \to G$. The hyperplane condition for transversals implies that $\pi_1,\dots,\pi_d$ are injections, hence bijections. Finally, the definition \eqref{LGpidef} of $L^d(G,\pi)$ implies \eqref{pi-sum}.

Conversely, given a solution $(\pi_1,\dots,\pi_d)$ to \eqref{pi-sum} with $\pi_1,\dots,\pi_d \in S$, consider the collection of entries of $L^d(G,\pi)$ indexed by \eqref{transversalx}. The condition \eqref{pi-sum} ensures that each symbol is represented exactly once (since $x$ is in position \eqref{transversalx}), and the condition that $\pi_1,\dots,\pi_d \in S$ ensures that no two entries occur in the same hyperplane.

For the last part of the lemma, just note that
\[
  \pi_1 + \cdots + \pi_d = \pi
\]
if and only if
\[
  \pi_1 \pi^{-1} \pi_{d+1} + \cdots + \pi_d \pi^{-1} \pi_{d+1} = \pi_{d+1}.
\]
Thus every solution to \eqref{pi-sum} corresponds to $n!$ solutions to \eqref{pi-sum-2}.
\end{proof}

Thus by applying Theorems~\ref{main-even} and~\ref{thm:main-4} we have the following.

\begin{theorem}\label{transversal-corollary}
Let $G$ be an abelian group of order $n$. Then the number of transversals in $L^d(G,\pi)$ is
\[
\begin{cases}
(e^{-1/2} + o(1)) n!^2/n^{n-1} & \text{if}~d=2~\text{and}~\Sigma G = 0,\\
(1 + O_d(n^{2-d})) n!^d/n^{n-1} & \text{if}~d\geq 3~\text{and}~(d+1)\Sigma G = 0,\\
0 & \text{otherwise}.
\end{cases}
\]
\end{theorem}

For context in the Latin squares literature, let $T(d,n)$ be the maximum number of transversals in a Latin cube of dimension $d$ and order $n$. Taranenko~\cite[Theorem~6.1]{taranenko} proved that
\[
  T(d,n) \leq \((1+o(1)) \frac{n^{d-1}}{e^d}\)^n.
\]
A matching lower bound was proved by Glebov and Luria~\cite[Theorems~1.3 and 4.1]{glebovluria} using a random construction: they proved that
\[
  T(d,n) \geq \((1 - o(1))\frac{n^{d-1}}{e^d}\)^n
\]
for all $n$ when $d=2$, and for infinitely many $n$ for each $d\geq 3$. Theorem~\ref{transversal-corollary} implies a more precise lower bound for $T(d,n)$, with an explicit Latin cube, whenever $d$ is odd or $n \not\equiv 2\pmod 4$. It is an amusing challenge to cover the remaining cases, but not one we entertain here.

\section{PRP-to-PRF conversion}
\label{sec:cryptography}

Two fundamental building blocks in cryptography are so-called \emph{pseudorandom functions} (PRFs) and \emph{pseudorandom permutations} (PRPs). Roughly speaking, a pseudorandom function is a family of functions $f_s:\{0,1\}^d \to \{0,1\}^d$ such that
\begin{itemize}
  \item for any $s$ and $x\in\{0,1\}^d$ there is an efficient algorithm for computing $f_s(x)$,
  \item if $s$ is chosen uniformly at random then it is not computationally practical to distinguish $f_s$ from a truly random function $\{0,1\}^d \to \{0,1\}^d$, with any statistical significance.
\end{itemize}
Similarly, a pseudorandom permutation is a family of permutation $f_s:\{0,1\}^d \to \{0,1\}^d$ such that
\begin{itemize}
  \item for any $s$ and $x\in\{0,1\}^d$ there is an efficient algorithm for computing $f_s(x)$,
  \item if $s$ is chosen uniformly at random then it is not computationally practical to distinguish $f_s$ from a truly random permutation $\{0,1\}^d \to \{0,1\}^d$, with any statistical significance.
\end{itemize}
Usually the domain of $s$ is taken to be $\{0,1\}^k$ for some $k$. In both cases we will write $F:\{0,1\}^k\times \{0,1\}^d \to \{0,1\}^d$ for the function defined by $F(s,x) = f_s(x)$, and we will refer to $F$ itself as the PRF/PRP. Both concepts are treated as cryptographic primitives, from which more complex cryptographic constructions are built, but they have somewhat different particular use cases. For more detail and background, see Goldreich~\cite[Chapter 3]{goldreich}.

The security of a PRF or a PRP is measured by a hypothetical attacker's \emph{advantage}, which is defined as follows. Suppose $A$ is some probabilistic algorithm which, given a function $f:\{0,1\}^d \to \{0,1\}^d$, outputs either $A(f) = 0$ or $A(f)=1$. Then the \emph{advantage} of $A$ as a PRF distinguisher is
\[
  \text{Adv}_F^\textup{PRF}(A) = |\P(A(f_s) = 1) - \P(A(f) = 1)|,
\]
where $s$ is drawn uniformly from $\{0,1\}^k$ and $f$ is an a genuinely uniform random function. The advantage to an attacker with access to at most $m$ queries is then defined as
\[
  \text{Adv}_F^\textup{PRF}(m) = \max \text{Adv}_F^\textup{PRF}(A),
\]
where the maximum is taken over all probabilitic algorithms $A$ which query at most $m$ values $f(x)$, where the places $x$ being queried may depend on previous values (an ``adaptive chosen plaintext attack''). Usually some constraint is also placed on the computational power of the attacker, but we can mostly focus on the number of queries. Advantage for PRP attackers is defined similarly.

Given a good PRF generator, a corresponding PRP generator can be constructed using the so-called Feistel cipher. The security of this construction has been extensively studied, starting with the seminal result of Luby and Rackoff~\cite{luby-rackoff}, which states that 3 rounds of the Feistel cipher are sufficient to guarantee security. More speficially, Luby and Rackoff proved that if a 3-round Feistel cipher is combined with an ideal PRF generator then
\[
  \text{Adv}_F^\text{PRP}(m) \leq m^2/2^{d/2}.
\]
Moreover we get stronger bounds if more rounds are used: see Patarin~\cite{patarin-LR}.

Conversion in the other direction is less well studied, partly because it often suffices just to consider a PRP itself as a PRF. Specifically, if we attempt to use a PRP generator $F$ itself as a PRF, then we have
\[
  \text{Adv}_F^\text{PRF}(m) \leq  \text{Adv}_F^\text{PRP}(m) +  m^2/2^d.
\]
However, such a construction is vulnerable to the so-called \emph{birthday attack}: one expects to have to query only about $2^{d/2}$ times before seeing a collision, so roughly $2^{d/2}$ queries should suffice to distinguish a PRP from a PRF. This attack shows that the above bound is close to sharp. We could overcome this problem by, say, doubling the size of $d$, but this comes at a cost in efficiency.

An alternative construction is to take two independent PRPs and use their bitwise xor as our PRF. In other words we take two independent pseudorandom permutations $\tilde{\pi}_1,\tilde{\pi}_2$ of $\F_2^d$ and we propose $\tilde{\pi}_1+\tilde{\pi}_2$ as a pseudorandom function. This was first analyzed by Bellare and Impagliazzo~\cite{bellare-impagliazzo} and Lucks~\cite{lucks}, and later by Patarin~\cite{patarin-xor} and others. An optimal security bound for this construction follows from Theorem~\ref{thm:L2}.

\begin{theorem}
Suppose we use the \textup{xor} construction to make a PRF generator $F$ from a PRP generator $E$. Then as long as $m < 2^d$ we have
\[
  \textup{Adv}_F^\text{PRF}(m) \leq 2\textup{Adv}_E^\text{PRP}(m) + O(m/2^{3d/2}).
\]
\end{theorem}
\begin{proof}
The following diagram may or may not be helpful:
\[
  \tilde\pi_1+\tilde\pi_2 \xleftrightarrow{\textup{Adv}_E^\textup{PRP}(m)} \pi_1 + \tilde\pi_2 \xleftrightarrow{\textup{Adv}_E^\textup{PRP}(m)} \pi_1 + \pi_2 \xleftrightarrow{O(m/2^{3d/2})} f.
\]
We are using $\tilde{\pi}_1 + \tilde{\pi}_2$ as our PRF, where $\tilde{\pi}_1$ and $\tilde{\pi}_2$ are independent draws from our PRP generator. The attacker cannot distinguish this from $\pi_1 + \pi_2$, where $\pi_1$ and $\pi_2$ are truly random permutations $\F_2^d\to\F_2^d$, with advantage greater than $2\textup{Adv}_E^\textup{PRP}(m)$. Since advantage is bounded by total variation distance, Theorem~\ref{thm:L2} implies that the attacker cannot distinguish $\pi_1+\pi_2$ from a truly random function $f$ with advantage greater than $O(m/2^{3d/2})$. Thus the claim follows from the triangle inequality for advantage.
\end{proof}

For a better introduction to PRP-to-PRF convesion, see Bellare, Krovetz, and Rogaway~\cite{bellare-krovetz-rogaway}.

\section*{Acknowledgements}

I am grateful to Freddie Manners and Rudi Mrazovi{\'c} for numerous technical discussions, to Samuel Neves for conversations about cryptography, and to Ian Wanless for conversations about Latin squares.

\bibliography{moreon}
\bibliographystyle{alpha}

\end{document}